\documentclass[reqno, a4paper, 10pt]{amsart}  
\usepackage[utf8]{inputenc}
\usepackage{epsfig} % for postscript graphics files
\usepackage{amsmath}
\usepackage{amssymb}
\usepackage{amsthm} 
\usepackage{accents}
\usepackage{url}
\usepackage{mathrsfs}
\usepackage{nicefrac}

\usepackage{moreverb}
\usepackage[colorlinks,bookmarksopen,bookmarksnumbered,citecolor=red,urlcolor=red]{hyperref}

\newcommand{\Ucal}{\mathcal{U}}

\newcommand{\Ycal}{\mathcal{Y}}

\newcommand{\Xcal}{\mathcal{X}}
\newcommand{\uc}{\bar{c}}
\newcommand{\lc}{\underbar{c}}

\renewcommand{\:}{\mathcal{\colon}}
\newcommand{\NN}{\mathbb{N}}

\newcommand{\RR}{\mathbb{R}}

\theoremstyle{plain}
\newtheorem{theorem}{Theorem}
\newtheorem{proposition}[theorem]{Proposition}

\newtheorem{lemma}[theorem]{Lemma}
\newtheorem{hypothesis}[theorem]{Hypothesis}
\theoremstyle{remark}
\newtheorem{remark}[theorem]{Remark}

\title[Relaxation for Hyperbolic PDE Mixed-Integer Optimal Control]{Relaxation Methods for Hyperbolic PDE Mixed-Integer Optimal Control Problems}

\author{Falk M. Hante$^\dag$}

\date{September 14, 2015}

\keywords{PDE-Constrained Optimization; Mixed-Integer Optimal Control; Hyperbolic PDEs}

\thanks{$^\dag$Falk M. Hante, Department Mathematik, Universität Erlangen-Nürnberg, Cauerstr. 11, 91058 Erlangen, Germany. E-mail: \url{falk.hante@fau.de}}

\begin{document}

%%%%%%%%%%%%%%%%%%%%%%%%%%%%%%%%%%%%%%%%%%%%%%%%%%%%%%%%%%%%%%%%%%%%%%%%%%%%%%%%
\begin{abstract} 
We extend the convergence analysis for methods solving PDE-constrained 
optimal control problems containing both discrete and continuous control decisions 
based on relaxation and rounding strategies to the class of first order semilinear
hyperbolic systems in one space dimension. The results are obtained by novel a-priori 
estimates for the size of the relaxation gap based on the characteristic flow,
fixed-point arguments and particular regularity theory for such mixed-integer 
control problems. As an application we consider a relaxation model for optimal 
flux switching control in conservation laws motivated by traffic flow problems.
\end{abstract}
%%%%%%%%%%%%%%%%%%%%%%%%%%%%%%%%%%%%%%%%%%%%%%%%%%%%%%%%%%%%%%%%%%%%%%%%%%%%%%%%

\maketitle

\section{Introduction}
Decision taking, e.\,g., specifying \textrm{on}/\textrm{off} for some fixed actuator on a plant, 
is a very elementary control mechanism. Critical infrastructure systems such as
gas pipeline networks, water canal networks or highway traffic networks are for instance 
controlled by switching valves \cite{PfetschEtAl2015}, weirs \cite{LamareGirardPrieur2015} or 
speed limit signs \cite{HeygiEtAl2005}, respectively. These decision are possibly to be taken 
in combination with determining additional continuously variable parameters such as the outlet 
pressure of a compressor in the gas network example. Further modeling aspects of such problems
are discussed in \cite{HLS2009}. We will refer to the open-loop optimization 
of such heterogeneous controls as a \emph{mixed-integer optimal control problem}. In the literature, 
such problems are also called \emph{discrete-continuous}, \emph{hybrid} or \emph{switching} 
\emph{optimal control problems}.  

A method for solving such problems numerically should yield integer feasible solutions
based on sufficiently accurate approximations of the plant's dynamics. Due to the combinatorial
complexity this is a difficult problem in general. In particular, it becomes a real challenge 
when the plant is modeled using partial differential equations (PDEs) as it is expedient for 
example in the mentioned infrastructure systems in nonstationary scenarios on large scale
networks. Thus, several different solution approaches have already been proposed in the literature.
They can be classified as total-discretization-, reformulation- and direct relaxation-based and
each of them have there assets and drawbacks.

Total discretization of the underlying dynamical system obviously leads to mixed-integer 
nonlinear programs (MINLPs). These become typically large, in particular in the PDE case. 
Hence, solving such MINLPs requires structure exploiting algorithms. If they are available,
such methods can provide global optimal solutions. This approach is followed for example 
in \cite{GeisslerMorsiSchewe2013} for gas networks or in \cite{GoettlichHertyZiegler2015} 
for traffic flow.

Reformulations turn the mixed-integer optimal control problem into---on some level---equivalent 
formulations which can then be solved with existing numerical methods.
For example, a variable time transformation method yields a continuous 
formulation \cite{LeeTeoRehbockJennings1999,Gerdts2006}. This approach is yet limited to problems 
governed by ordinary differential equations (ODEs). Optimal switching controls can in some 
cases be sophisticatedly characterized using the viscosity solution of a Hamilton-Jacobi-Bellman system, 
see \cite{CapuzzoDolcettaEvans1984} for ODE problems and \cite{Yong1989}
for a generalization to PDEs. However, in particular in the PDE case, the 
Hamilton-Jacobi-Bellman system is again typically large and numerically a difficult problem.
Complementarity-based reformulations are used in \cite{BaumruckerBiegler2009} and 
have been applied to gas network optimization in \cite{BaumruckerBiegler2010}.
The resulting system then require special nonlinear solvers, or additional relaxation techniques.

Direct relaxation of the decision variables obviously turns the problem into a fully continuous setting.
The resulting problem can then be solved with well-established nonlinear optimal control 
techniques and integer solutions can be recovered from rounding. For ODE problems, it has been
shown that the relaxation gap, i.\,e., the distance of the optimal value for the relaxed problem 
and the one corresponding to the rounded solution, can be made arbitrary small using relaxation for
a convexified problem together with a constructive rounding strategy \cite{SagerBockDiehl2012}. 
The so obtained epsilon-optimal controls may involve frequent switching, but suboptimal solutions
can still be obtained including combinatorial constraints limiting for example the number of switches
\cite{SagerJungKirches2011}. The approach has also been generalized to dealing with vanishing 
constraints \cite{JungKirchesSager2013} and to mixed-integer control problems for abstract semilinear 
evolutions on Banach spaces using semigroup theory \cite{HanteSager2013}. 
The generalization to Banach spaces covers semilinear PDE cases, but the analysis in \cite{HanteSager2013} 
assumes a certain degree of smoothness of the solution. Due to the inherent discontinuities of the 
integer control interfering with the solution, these regularity assumptions are only known to be valid 
in parabolic cases due to the natural smoothing properties of ellipic operators (examples are 
discussed \cite{HanteSager2013}). In particular, these regularity assumptions are known to fail for
classical solutions and are in general very difficult to be verified for weak solutions 
in hyperbolic cases. Unfortunately, all examples mentioned at the beginning naturally involve 
hyperbolic PDEs which are at present not supported by the available theory for this approach.

The main contribution of this paper is to extend the analysis for the direct relaxation approach to semilinear
hyperbolic PDEs with distributed mixed-integer control. Instead of employing semigroup theory on classical Sobolev spaces 
as in \cite{HanteSager2013}, we obtain novel a-priori estimates on the size of the relaxation gap using the 
method of characteristics and particular regularity results on semi-classical Sobolev spaces. This is an 
important step to establish relaxation techniques for optimization of the mentioned infrastructure systems. 
For instance, transient gas network operation can often be accurately modeled using semilinear 
Euler-equations \cite{BandaHerty2008}. For an application of the results obtained in this paper to gas-network 
optimization, valve switching is then to be modeled as a distributed control on very small pipe sections. 
This is intended as future work. In this paper, we will discuss the application of the method to a relaxation 
model of nonlinear conservation laws motivated by traffic flow control. 

The paper is organized as follows. In Section~\ref{sec:problem} we give a detailed problem formulation for 
a semilinear hyperbolic mixed-integer optimal control problem and relate it to a relaxed and convexified problem.
In Section~\ref{sec:estimates} we estimate the gap made by this approach in terms of the integrated difference 
of two controls. The latter quantity is not a norm on the respective control spaces which make 
the estimates technically difficult. The estimates yield a convergence result for the relaxation method in 
Proposition~\ref{prop:relgap}. In Section~\ref{sec:example}, we employ the relaxation approach to 
the already mentioned example motivated by traffic flow, where we use an adjoint-equation based gradient-decent
algorithm to compute solutions to a relaxed problem and we apply rounding strategies in order to verify the
predicted convergence numerically. In Section~\ref{sec:conclusion}, we draw a brief conclusion.

\section{Problem Formulation and the Relaxation Approach}\label{sec:problem}
For some real constants $L,T>0$, a natural number $n$, a diagonal matrix function 
$\Lambda=\text{diag}(\lambda_1,\ldots,\lambda_n)$ with $\lambda_i\: [0,L] \to \RR$, 
$i\in\{1,\ldots,n\}$, normed vector spaces $U$ and $V$ and a nonlinear function 
$f\: \RR^n \times U \times V \to \RR^n$, we consider a controlled system of semilinear 
hyperbolic PDEs in two variables $t$ (time) and $x$ (a single space variable)
\begin{equation}\label{eq:SysPDE}
  y_t + \Lambda(x) y_x = f(y,u(t),v(t))\quad\text{on}~(0,T) \times (0,L)
\end{equation}
for an unkown vector function $y=(y_1,\ldots,y_n)^\top \: (0,T) \times (0,L) \to \RR^n$ and
two controls $u\: [0,T] \to U$ and $v\: [0,T] \to V$. We assume that, for some $r \in \{1,\ldots,n\}$,
$\lambda_i<0$ for $i\in\{1,\ldots,r\}$ and $\lambda_i>0$ for $i=\{r+1,\ldots,n\}$
and set $y^-=(y_1,\ldots,y_r)^\top$, $y^+=(y_{r+1},\ldots,y_n)^\top$, so that 
$y=(y^-,y^+)^\top$. Further, we consider \eqref{eq:SysPDE} subject to boundary conditions
\begin{equation}\label{eq:SysBC}
\begin{pmatrix}y^-(t,L)\\y^+(t,0)\end{pmatrix} = \begin{pmatrix}G_{--} & G_{-+}\\ G_{+-} & G_{++} \end{pmatrix} \begin{pmatrix}y^-(t,0)\\y^+(t,L)\end{pmatrix} + \begin{pmatrix} d^-(t) \\ d^+(t) \end{pmatrix},\quad \text{on}~(0,T)
\end{equation}
and an initial condition
\begin{equation}\label{eq:SysIC}
 y(0,x)=\bar{y}(x)\quad \text{on}~[0,L],
\end{equation}
where we assume that $G_{--} \in \RR^{r \times r}$, $G_{-+} \in \RR^{r \times (n-r)}$, $G_{+-} \in \RR^{(n-r) \times r}$,
$G_{++} \in \RR^{(n-r) \times (n-r)}$, $d^-\: [0,T] \to \RR^r$, $d^+\: [0,T] \to \RR^{(n-r)}$, and 
$\bar{y}\: [0,L] \to \RR^n$. The diagonal form in \eqref{eq:SysPDE} is without loss of generality, because it can be achieved by a change of 
variables for many physical systems, see e.\,g., \cite{CourantHilbert1962,Bressan2000}.

Our interest will be to find controls $u$ and $v$ (piecewise smooth and piecewise constant, respectively) 
that minimize a cost criterion
\begin{equation}\label{eq:Cost}
   J(y(T,\cdot))=\int_0^L g(y(T,x))\,dx,
\end{equation}
where $g\: \RR^n \to \RR$ is a non-linear function, subject to the ordinary control constraint that
\begin{equation}\label{eq:Cconstraint}
 u(t) \in \Ucal
\end{equation}
for some closed (and otherwise arbitrary) set $\Ucal \subset U$ and 
the discrete control constraint
\begin{equation}
 v(t) \in \{v^1,v^2,\ldots,v^M\}
\end{equation}
for some given values $v^1,v^2,\ldots,v^M \in V$, $M \in \NN$. 

To this end, we note that for any piecewise smooth $u$ and any piecewise constant $v$, the equation \eqref{eq:SysPDE} 
is equivalent to
\begin{equation}\label{eq:SysPDE2}
  y_t + \Lambda y_x = \sum_{j=1}^M \alpha_j f(y,u,v^j)\quad\text{on}~(0,T) \times (0,L)
\end{equation}
with new controls $\alpha_j=\alpha_j(t) \in \{0,1\}$, $j=1,\ldots,M$, satisfying $\sum_{j=1}^M \alpha_j(t)=1$
for almost every $t$. We will in the following estimate the gap (measured in terms of the norm of $y$) of
weak solutions for \eqref{eq:SysPDE2} (hence solutions to the original problem) and the relaxed problem
\begin{equation}\label{eq:SysPDErelaxed}
  y_t + \Lambda y_x = \sum_{j=1}^M \beta_j f(y,u,v^j)\quad\text{on}~(0,T) \times (0,L)
\end{equation}
with controls $\beta_j=\beta_j(t) \in [0,1]$, $j=1,\ldots,M$, satisfying $\sum_{j=1}^M \beta_j(t)=1$
for almost every $t$. Note that we only replaced the condition ``new controls in $\{0,1\}$'' with
the condition ``new controls in $[0,1]$''. Further, observe that the minimization of \eqref{eq:Cost} subject to
the relaxed problem \eqref{eq:SysPDErelaxed} with the boundary data \eqref{eq:SysBC} and \eqref{eq:SysIC}
and control constraints 
\begin{equation}\label{eq:Cconstraintrelaxed}
 \eqref{eq:Cconstraint},\quad \beta_j(t) \in [0,1],~j=1,\ldots,M\quad\text{and}\quad \sum_{j=1}^M \beta_j(t)=1,\quad t \in (0,T)~\text{a.\,e.}
\end{equation}
is an optimal control problem of a semilinear hyperbolic system in standard form which can be solved efficiently with 
existing numerical methods \cite{HinzeEtAl2009,BandaHerty2009}.

The main motivation for the a-priori estimates derived in Section~\ref{sec:estimates} is the following observation proved 
in \cite{SagerBockDiehl2012}.

\begin{lemma}\label{lem:SUR} Let $\beta\: [0,T] \to [0,1]^M$ be a measurable function such that
$\sum_{j=1}^M \beta_j(t)=1$ for almost every $t$. Then, there exists a piecewise constant function 
$\alpha\: [0,T] \to \{0,1\}^M$ satisfying $\sum_{j=1}^M \alpha_j(t)=1$ for all $t \in [0,T]$ such that
\begin{equation} \label{eq:SURbound}
 \max_{j=1,\ldots,M} \sup_{t \in [0,T]} \left| \int_0^t \beta_j(s) - \alpha_j(s) \,ds \right| \leq (M-1)\Delta t.
\end{equation}
where $\Delta t$ is the length of the largest subinterval where $\alpha$ is taken constant.
\end{lemma}

\begin{remark} The proof in \cite{SagerBockDiehl2012} for the existence of $\alpha$ as in the previous Lemma 
is constructive, giving rise to a numerical method that was also explicitly used in \cite{HanteSager2013}. 
Moreover, it has recently been show that the right hand side in \eqref{eq:SURbound} 
can been improved to $\mathcal{O}(\log(M)) \Delta t$ \cite{Kirches2015}.
\end{remark}

We are therefore interested in estimates on the relaxation gap in terms of the left hand side of \eqref{eq:SURbound}.
Note that this quantity is not a norm and that these estimates therefore are neither classical nor obvious. For our
analysis, we will consider solutions of \eqref{eq:SysPDE} or \eqref{eq:SysPDErelaxed}, respectively, subject to the 
boundary conditions \eqref{eq:SysBC} and the initial condition \eqref{eq:SysIC} in the sense
of the forward characteristic flow. Letting $\Omega_T=[0,T] \times [0,L]$ and $s_i(t;\tau,\sigma)$, $i\in\{1,\ldots,n\}$ denote 
the characteristic curves defined by
\begin{equation}
 \frac{d}{dt}s_i = \lambda_i(s_i),\quad s_i(\tau;\tau,\sigma)=(\tau,\sigma),\quad (\tau,\sigma) \in \Omega_T
\end{equation}
the forward characteristic flow of \eqref{eq:SysPDErelaxed} is obtained as a fixed point $y\: \Omega_T \to \RR^n$
of the following integral transformation
% The $L^1$-solution of \eqref{eq:SysPDErelaxed} is given as a fixed point on $\Ycal:=C([0,T];L^1([0,L];\RR^n))$ of an integral transformation 
\begin{equation}\label{eq:psidef1}
 \psi(y)(\tau,\sigma)=(\psi_1(y)(\tau,\sigma),\ldots,\psi_n(y)(\tau,\sigma))
\end{equation}
with $\psi_i(y)(\tau,\sigma)$ defined recursively as
\begin{equation}\label{eq:psidef2}
 \psi_i(y)(\tau,\sigma) = Y_i(\psi;\tau,\sigma) + \sum_{j=1}^M \int_{t_i^*}^\tau \beta_j(t) f(y(t,s_i(t;\tau,\sigma)),u(t),v^j)\,dt,
\end{equation}
where $t_i^*=t_i^*(\tau,\sigma)$ denotes the intersection time of the curve $s_i(\cdot;\tau,\sigma)$ with the boundary of $\Omega_T$
backward in time and $Y_i(\psi;\tau,\sigma)$ is defined as the $i$-th component of 
\begin{equation}\label{eq:psidef3}
\begin{cases} 
 \begin{pmatrix}
 G_{--} & G_{-+}\\ 
 G_{+-} & G_{++} 
 \end{pmatrix} 
 \begin{pmatrix}
   \psi^-(y)(t_i^*,0)\\
   \psi^+(y)(t_i^*,L)
 \end{pmatrix} + \begin{pmatrix} d^-(t_i^*) \\ d^+(t_i^*) \end{pmatrix}\quad &\text{if $t_i^*>0$},\\
 \bar{y}(t_i^*,s_i(t_i^*;\tau,\sigma))\quad&\text{if $t_i^*=0$},
\end{cases}
\end{equation}
where $\psi^-=(\psi_1,\ldots,\psi_r)^\top$ and $\psi^+=(\psi_{r+1},\ldots,\psi_{n})^\top$. The so defined solution coincides with
the usual weak solution of \eqref{eq:SysPDE} or \eqref{eq:SysPDErelaxed}, respectively, subject to the boundary conditions \eqref{eq:SysBC} 
and the initial condition \eqref{eq:SysIC} in appropriate spaces \cite{HallerHoermann2008}.

\section{A-priori Estimates on the Relaxation Gap}\label{sec:estimates}
The subsequent analysis is based on a particular regularity result for hyperbolic initial boundary value problems obtained 
in \cite{Oberguggenberger1986}, giving a sufficient condition for the solution defined by the forward characteristic flow 
having bounded distributional derivatives along almost every characteristic curve if the initial and boundary data is 
piecewise smooth. To be more specific, for any $p,q,\mu \in \NN$ and any family of disjoint open sets $\Omega_m \subset \RR^p$, $m=1,\ldots,\mu$, 
we will denote by
\begin{equation}
 \bigotimes_{m=1}^\mu W^{1,1}(\Omega_m;\RR^q),\qquad \bigotimes_{m=1}^\mu C^{0}(\Omega_m;\RR^q)
\end{equation}
the set of functions $h$ defined on the closure of $\bigcup_{m=1}^\mu \Omega_m$ with image in $\RR^q$ so that their restriction to $\Omega_m$
belongs to the classical Sobolev space $W^{1,1}(\Omega_m;\RR^q)$ or the Banach space $C^0(\Omega_m;\RR^q)$, respectively.

Then, we make the following assumptions.
\begin{hypothesis}\label{hyp:Reg}
The components of $\Lambda$, $\lambda_i: [0,L] \to \RR$, $i\in\{1,\ldots,n\}$ are Lipschitz-continuous. The functions 
$y \to f(y,u,v^j)$ are smooth and satisfy $f(0,u,v^j)=0$ for all $j=1,\ldots,M$ and $u \in U$. The function 
$u \to f(y,u,v^j)$ is locally Lipschitz-continuous on $U$ for all $j=1,\ldots,M$ and $y \in Y$.
Further, for all $T>0$, there exist finitely many points $0=\tau_0<\tau_1<\ldots<\tau_{K-1}<\tau_K=T$ and 
$0=x_0<x_1<\ldots<x_{\nu-1}<x_\nu=L$ so that the initial and boundary data satisfies 
$\bar{y} \in \bigotimes_{i=1}^\nu W^{1,1}(x_i,x_{i+1};\RR^n)$, $d^- \in \bigotimes_{i=1}^K W^{1,1}(\tau_i,\tau_{i+1};\RR^r)$, 
and $d^+ \in \bigotimes_{i=1}^K W^{1,1}(\tau_i,\tau_{i+1};\RR^{(n-r)})$.
\end{hypothesis}

These assumptions yield the following wellposedness result.

\begin{lemma}\label{lem:L1sol} Under Hypothesis~\ref{hyp:Reg} and for any piecewise smooth control $u\: [0,\infty) \to U$ 
there exist constants $T,K>0$ such that \eqref{eq:SysPDErelaxed} subject to the boundary conditions \eqref{eq:SysBC} 
and the initial condition \eqref{eq:SysIC} admits a unique solution $y$ in $C([0,T];L^\infty(0,L;\RR^n)\cap L^1(0,L;\RR^n))$
for all piecewise smooth controls $\beta\: [0,T] \to [0,1]^M$. In particular, this solution is given as the (unique) 
fixed point of \eqref{eq:psidef1}--\eqref{eq:psidef3} resulting from a strict contraction with contraction constant $\frac12$ 
on $\Ycal = C([0,T];L^1(0,L;\RR^n))$ equipped with the norm
\begin{equation}\label{eq:defDagnorm}
\| y \|_\dag=\sup_{t \in [0,T]} e^{-Kt} \sum_{i=1}^n \int_0^L |y_i(t,x)|dx.
\end{equation}
\end{lemma}
\begin{proof} Hypothesis~\ref{hyp:Reg} implies that the initial and boundary data, respectively, 
satisfies $\bar{y} \in L^p(0,L;\RR^n)$, $d^- \in L^p(0,T;\RR^{r})$ and $d^+ \in L^p(0,T;\RR^{(n-r)})$
for all $p\in\{1,\infty\}$ and all $T>0$ and that $f$ is locally Lipschitz-continuous in $y$. 
Noting that all feasible $\beta$ are uniformly bounded by one in the sup-norm, the result 
follows from classical fixed-point arguments \cite{CourantHilbert1962}, \cite{Bressan2000}.
\end{proof}

In order to obtain further regularity properties of the solution, we adapt the following definitions from 
\cite{Oberguggenberger1986}. For any piecewise smooth control $\beta\: [0,\infty) \to [0,1]^M$ with
discontinuities at $\{\theta_i\}_{i=1}^\infty$, let $E^0(\beta)$ denote the union of all forward 
characteristic curves $s_i$ generated by $\Lambda$  (and their reflections at the boundaries) which 
lie in $\Omega_\infty = [0,\infty) \times [0,L]$, which emerge from the boundary points
\begin{equation}\label{eq:Bpoints}
 \{(0,x_0),\ldots,(0,x_\nu),~(\tau_0,0),\ldots,(\tau_K,0),~(\tau_0,L),\ldots,(\tau_K,0)\}
\end{equation}
and their intersection points with the sets $\{\theta_i\} \times [0,L]$ for all $i=1,\ldots,\infty$.
Further, for $k\in\NN$, let $E^k(\beta)$ be the union of $E^{k-1}(\beta)$ and all forward characteristic 
curves (and their boundary reflections) emerging from intersection points of characteristics that defined $E^{k-1}$.
Let $E_T(\beta)$ be the closure of all points in $\bigcup_{k=1}^\infty E^k(\beta) \cap \Omega_T$.
Supposing that
\begin{equation}\label{eq:HypNoDense}
 E_T(\beta) \cap \{t\} \times [0,L]~\text{is nowhere dense in}~\{t\} \times [0,L],\quad t\in (0,T],
\end{equation}
$E_T(\beta)$ is defined by finitely many discrete curves, which divide $\Omega_T$ up into finitely many simply 
connected open sets $\Omega_T^m$, $m=1,\ldots,\mu$, $\mu \in \NN$. We set
\begin{equation}
W^{1,1}_*(\Omega_T \setminus E_T(\beta)) := \bigotimes_{m=1}^\mu W^{1,1}(\Omega_T^m),\quad C^0_*(\Omega_T \setminus E_T(\beta)) := \bigotimes_{m=1}^\mu C^0(\Omega_T^m).
\end{equation}

With this notation, we note the following additional regularity of the $L^1$-solution for piecewise smooth data and controls.

\begin{lemma} \label{lem:Oberguggenberger} Under Hypothesis~\ref{hyp:Reg} and for any piecewise smooth control $u\: [0,\infty) \to U$ 
and any piecewise smooth control $\beta\: [0,\infty) \to [0,1]^M$, there exist a constant $T>0$ (less or equal to the constant obtained in 
Lemma~\ref{lem:L1sol}) such that \eqref{eq:HypNoDense} holds for the set $E_T(\beta)$ defined above. Moreover, it holds that the solution 
$y$ of \eqref{eq:SysPDErelaxed} subject to the boundary conditions \eqref{eq:SysBC} and the initial condition \eqref{eq:SysIC} 
obtained in Lemma~\ref{lem:L1sol} satisfies
\begin{equation}
y \in W^{1,1}_*(\Omega_T \setminus E_T(\beta)) \cap C^0_*(\Omega_T \setminus E_T(\beta)).
\end{equation}
\end{lemma}
\begin{proof}
The result follows directly from \cite[Theorem~3.1 and Remark~3.3]{Oberguggenberger1986} applied iteratively on time intervals 
where $d^-$, $d^+$, $u$ and $\beta$ are jointly continuous.
\end{proof}

Moreover, we will need the following fixed point argument.

\begin{lemma}\label{lem:fixpointarg}
Consider a Banach space $(X,\|\cdot\|_X)$. Let $\Xcal := \{ x \in X : \|x\|_\dag < \infty\}$ for 
some norm $\|\cdot\|_\dag$ satisfying
\begin{equation}\label{eq:normcomp}
 \lc \|x\|_X \leq \|x\|_\dag \leq \uc \|x\|_X,\quad x \in X.
\end{equation}
Let $\phi,\psi\: \Xcal \to \Xcal$ be continuous mappings such that
\begin{equation}\label{eq:PhiPsiconstraction}
\|\phi(x) - \phi(\zeta)\|_\dag \leq \frac12 \| x - \zeta \|_\dag, \quad \| \psi(x) -  \psi(\zeta)\|_\dag \leq \frac12 \| x - \zeta \|_\dag,\quad x,\zeta \in \Xcal.
\end{equation}
Suppose that there exists a constant $C>0$ so that
\begin{equation}\label{eq:PhiPsibound}
\|\phi(x^*)-\psi(x^*)\|_X \leq C,
\end{equation}
where $x^* \in \Xcal$ denotes the (unique) fixed-point of $\phi$ in $X$. For the (unique) fixed point $\zeta^* \in \Xcal$ of $\psi$ in $X$ it then holds
\begin{equation}\label{eq:conclusionfixedpointest}
 \|x^* - \zeta^*\|_X \leq 2 \lc^{-1} \uc C. 
\end{equation}
\end{lemma}
\begin{proof}
By the triangular inequality and \eqref{eq:normcomp} we have
\begin{equation}
 \| x^* - \zeta^*\|_\dag \leq \uc \| \phi(x^*) - \psi(x^*)\|_X + \|\psi(x^*) - \psi(\zeta^*)\|_\dag. 
\end{equation}
Using \eqref{eq:PhiPsiconstraction} and \eqref{eq:PhiPsibound}, this yields
\begin{equation}
 \| x^* - \zeta^*\|_\dag \leq \uc C + \frac12 \|x^* - \zeta^*\|_\dag,
\end{equation}
and equivalently \eqref{eq:conclusionfixedpointest} using again \eqref{eq:normcomp}.
\end{proof}

With the above auxilary results, we can now prove the following a-priori estimate on the difference of two solutions 
corresponding to two different controls with a bounded integerated difference.

\begin{theorem}\label{thm:gapest} Assume Hypothesis~\ref{hyp:Reg} and let $u\: [0,\infty) \to U$, $\beta\: [0,\infty) \to [0,1]^M$ 
and $\tilde{\beta}\: [0,\infty) \to [0,1]^M$ be piecewise smooth controls. Moreover, assume that for some given $\varepsilon>0$ and 
$T$ sufficiently small,
\begin{equation}\label{eq:SURboundTHM1}
 \max_{j=1,\ldots,M} \sup_{t \in [0,T]} \left| \int_0^t \beta_j(s) - \tilde{\beta}_j(s) \,ds \right| \leq \varepsilon.
\end{equation}
Then, there exists a constant $\bar C=\bar C(u,\beta)>0$ (independent of $\tilde \beta$) so that
\begin{equation}\label{eq:ToShowTheorem}
 \|y(u,\beta)-y(u,\tilde\beta)\|_Y \leq \bar C \varepsilon,
\end{equation}
where $y(u,\beta)$ and $y(u,\tilde{\beta})$ denote the $L^1$-solutions of the relaxed problem \eqref{eq:SysPDErelaxed} subject 
to the boundary conditions \eqref{eq:SysBC}, initial condition \eqref{eq:SysIC} and the control constraints \eqref{eq:Cconstraintrelaxed} 
with controls $(u,\beta)$ and $(u,\tilde{\beta})$, respectively. 
\end{theorem}
\begin{proof}
Let $T$ be less or equal to the constant $T$ obtained in Lemma~\ref{lem:Oberguggenberger} for the fixed control $(u,\beta)$. Moreover, 
let $\psi$ denote the integral transformation \eqref{eq:psidef1}--\eqref{eq:psidef3} associated to the fixed control $(u,\beta)$ and let 
$\phi$ denote the corresponding integral transformation associated to the fixed control $(u,\tilde\beta)$. By Lemma~\ref{lem:L1sol}, there 
exists $K>0$ such that both transformations are strict contractions with contraction constant $\frac12$ on $(\Ycal)$ with respect 
to the norm $\|\cdot\|_\dag$ defined in \eqref{eq:defDagnorm}, possessing the unique fixed points $y=y(u,\beta)$ and 
$\tilde y = y(u,\tilde{\beta})$, respectively.

We will now show existence of a constant $\tilde{C}$ such that
\begin{equation}\label{eq:toshowNew}
\int_0^L |\psi_i(y)(\tau,\sigma)-\phi_i(y)(\tau,\sigma)| d\sigma \leq \tilde C \varepsilon,~i=1,\ldots,n, ~t \in [0,T]
\end{equation}
holds for the fixed point $y=y(u,\beta)$.

We have for $\tau \in [0,T]$ and $0 \leq t^*_i=t^*_i(\tau,\sigma)<\tau$
\begin{equation}
\begin{aligned}
&\int_0^L \left| \psi_i(y)(\tau,\sigma)-\phi_i(y)(\tau,\sigma)\right|d\sigma \leq  \|Y_i(\psi;\tau,\sigma)-Y_i(\phi;\tau,\sigma)\|~+\\
&\qquad \sum_{j=1}^M \int_0^L \left| \int_{t_i^*}^\tau [\beta_j(\vartheta)-\tilde\beta_j(\vartheta)] f(y(\vartheta,s_i(\vartheta;\tau,\sigma)),u(\vartheta),v^i)\,d\vartheta\right|d\sigma.
\end{aligned}
\end{equation}

Integration by parts for the integral in $\vartheta$ yields
\begin{equation}
\begin{aligned}
& \sum_{j=1}^M \int_0^L \left| \int_{t_i^*}^\tau [\beta_j(\vartheta)-\tilde\beta_j(\vartheta)] f(y(\vartheta,s_i(\vartheta;\tau,\sigma)),u(\vartheta),v^j)\,d\vartheta\right|d\sigma \leq\\
&\quad \sum_{j=1}^M \int_0^L  \int_{t_i^*}^\tau \left| \int_0^\vartheta \beta_j(\xi)-\tilde\beta_j(\xi)\,d\xi\right| \left| D_\vartheta f(y(\vartheta,s_i(\vartheta;\tau,\sigma)),u(\vartheta),v^j)\right| d\vartheta d\sigma~+ \\
&\quad \sum_{j=1}^M \int_0^L \left| \int_{0}^\tau \beta_j(\xi)-\tilde \beta_j(\xi)\,d\xi \right| |f(y(\tau,s_i(\tau;\tau,\sigma)),u(\tau),v^j)|d\sigma~+ \\
&\quad \sum_{j=1}^M \int_0^L \left| \int_0^{t_i^*} \beta_j(\xi)-\tilde \beta_j(\xi)\,d\xi \right| |f(y(t_i^*,s_i(t_i^*;\tau,\sigma)),u(t_i^*),v^j)|d\sigma.
\end{aligned}
\end{equation}
By assumption \eqref{eq:SURboundTHM1}, this estimate becomes
\begin{equation}\label{eq:proofboundstep}
\begin{aligned}
& \sum_{j=1}^M \int_0^L \left| \int_{t_i^*}^\tau [\beta_j(\vartheta)-\tilde\beta_j(\vartheta)] f(y(\vartheta,s_i(\vartheta;\tau,\sigma)),u(\vartheta),v^j)\,d\vartheta\right|d\sigma \leq\\
&\quad \varepsilon \sum_{j=1}^M \bigg(\int_{t_i^*}^\tau \int_0^L \left| D_{\vartheta} f(y(\vartheta,s_i(\vartheta;\tau,\sigma)),u(\vartheta),v^j)\right| d\sigma~+ \\
&\quad \int_0^L |f(y(\tau,s_i(\tau;\tau,\sigma)),u(\tau),v^j)| + |f(y(t_i^*,s_i(t_i^*;\tau,\sigma)),u(t_i^*),v^j)| d\sigma \bigg).
\end{aligned}
\end{equation}
The chain rule yields
\begin{equation}\label{eq:proofboundstep2}
\begin{aligned}
&\int_0^L \left| D_{\vartheta} f(y(\vartheta,s_i(\vartheta;\tau,\sigma)),u(\vartheta),v^j)\right| d\sigma  \leq\\
&\qquad\sum_{k=1}^n\int_0^L \bigg| D_{y_k} f(y(\vartheta,s_i(\vartheta;\tau,\sigma)),u(\vartheta),v^j)\bigg| \bigg|D_\vartheta y_k(\vartheta,s_i(\vartheta;\tau,\sigma))\bigg|~+ \\ 
&\qquad \bigg|D_{u} f(y(\vartheta,s_i(\vartheta;\tau,\sigma)),u(\vartheta),v^j)\bigg| \bigg|D_\vartheta u(\vartheta) \bigg| d\sigma.
\end{aligned}
\end{equation}
The assumptions in Hypothesis~\ref{hyp:Reg} on $f$, the assumed regularity of $u$ and the regularity result on $y=y(u,\beta)$ 
in Theorem~\ref{lem:Oberguggenberger} now yields that all terms in absolute values appearing in the right hand side of \eqref{eq:proofboundstep2} are essentially bounded. 
Since all remaining terms in the right hand side of \eqref{eq:proofboundstep} are finite and independent of $\tilde \beta$, there exists a constant $\hat{C}$ independent of $\tilde \beta$ such that
\begin{equation}
 \sum_{j=1}^M \int_0^L \left| \int_{t_i^*}^\tau [\beta_j(\vartheta)-\tilde\beta_j(\vartheta)] f(y(\vartheta,s_i(\vartheta;\tau,\sigma)),u(\vartheta),v^j)\,d\vartheta\right|d\sigma \leq \hat{C}\varepsilon.
\end{equation}
Moreover, for $\|Y_i(\psi;\tau,\sigma)-Y_i(\phi;\tau,\sigma)\|$, we can use \eqref{eq:psidef3}. For $i \in \{1,\ldots,r\}$ and $(\tau,\sigma)$ such that $t_i^*>0$ for all $\sigma \in [0,L]$, we obtain
\begin{equation}\label{eq:BCest}
 \|Y_i(\psi;\tau,\sigma)-Y_i(\phi;\tau,\sigma)\| \leq \|G\| \left\| \begin{pmatrix} \psi^-(y)(t_i^*,0) - \phi^-(y)(t_i^*,0)\\ \psi^+(y)(t_i^*,0) - \phi^+(y)(t_i^*,0) \end{pmatrix} \right\|
\end{equation}
and a similar estimate is obtained in the other cases. We can then repeat the above estimates for each of the components of the vectors in the right hand side of \eqref{eq:BCest}. 
Since $t_i^*<\tau$ and since the estimate becomes trivial for $t_i^*=0$, the existence of a constant $\tilde{C}$ independent of $\tilde \beta$ 
such that \eqref{eq:toshowNew} follows from induction.

Finally, we note that $X=C([0,T];L^1(0,L;\RR^n))$ is a Banach space for which \eqref{eq:normcomp} holds with $\lc=e^{-KT}$ and $\uc=1$. So the desired estimate \eqref{eq:ToShowTheorem} follows \eqref{eq:toshowNew} and Lemma~\ref{lem:fixpointarg} with $\Xcal=\Ycal$ and $\bar C = 2 \lc^{-1}\uc n \tilde C = 2e^{KT} n \tilde C$.
\end{proof}

A combination of Theorem~\ref{thm:gapest} and Lemma~\ref{lem:SUR} yields a method for solving the mixed-integer optimal control problem stated 
in Section~\ref{sec:problem} subject to Hypothesis~\ref{hyp:Reg} up to any requested accuracy. We state this method in form of the 
following proposition and remarks.

\begin{proposition} \label{prop:relgap} Assume Hypothesis~\ref{hyp:Reg} and suppose that for some sufficiently small $T$ the minimization of \eqref{eq:Cost} subject to the relaxed problem \eqref{eq:SysPDErelaxed} with the boundary data \eqref{eq:SysBC} and \eqref{eq:SysIC} and control constraints \eqref{eq:Cconstraintrelaxed} admits a piecewise smooth optimal control $(u,\beta)$ with associate optimal value $J^*$. Let $\alpha^k$ be a sequence of controls $\alpha^k\: [0,T] \to \{0,1\}$ given by Lemma~\ref{lem:SUR} obtained from $\beta$ for a sequence $\Delta t^k \to 0$ as $k \to \infty$. Define $v^k \: [0,T] \to V$ by $v^k(t)=\sum_{i=1}^M \alpha^k(t)v^k$. Then $y(u,v^k)$ converges strongly to $y(u,\beta)$ as $k\to \infty$. In particular, for norm-continuous cost functions $J$,
\begin{equation}\label{eq:convergence}
 \lim_{k \to \infty} J(y(u,v^k))=J^*.
\end{equation}
For Lipschitz-continuous $J$, the convergence speed is at least linear in $\Delta t^k$.
\end{proposition}
\begin{proof}
This follows from applying Lemma~\ref{lem:SUR} with $\alpha=\alpha^k$ and Theorem~\ref{thm:gapest} with $\tilde \beta = \alpha^k$ and $\varepsilon=(M-1)\Delta t^k$
for each fixed $k$.
\end{proof}

\begin{remark}\label{rem:subopt} We may relax the assumption of existence of a piecewise smooth optimal control $(u,\beta)$ in Proposition~\ref{prop:relgap}
by existence of a piecewise smooth suboptimal control $(u,\beta)$ with $J^*=J(y(u,\beta))$. The statement of the Proposition then does not change otherwise, but it clearly 
yields a slightly weaker conclusion. We may also include combinatorial constraints on the integer controls, e.\,g., bounding the number of switches between the modes. Then, we do in 
general not obtain the convergence \eqref{eq:convergence}, but the relaxation method then still yields asymptotically a monotonic decreasing sequence $J(y(u,v^k))$ of suboptimal 
solutions as $k \to \infty$. Moreover, the relaxation method can also deal with state constraints. Details of these extensions are discussed in \cite{HanteSager2013}.
\end{remark}

\begin{remark}\label{rem:Edense}
Note that Theorem~\ref{thm:gapest}, Proposition~\ref{prop:relgap} and Remark~\ref{rem:subopt} are stated for sufficiently small $T>0$ guaranteeing the existence of the solution for the control $(u,\beta)$ both in the sense of Lemma~\ref{lem:L1sol} and Lemma~\ref{lem:Oberguggenberger}. However, this limitation can be removed under additional assumptions 
on $f$, $\Lambda$ and the system's dimension. For example, is well-known that if the mappings $y \mapsto f(y,u,v^j)$ and $x \mapsto \Lambda(x)$ are Lipschitz-continuous, then the $L^1$-solution (coinciding with the $L^\infty$-solution) in the sense of Lemma~\ref{lem:L1sol} exists for all $T>0$. Further, in the important case of $n=2$, we have $E_T(\beta)=E^0_T(\beta)$ 
and hence $E_T(\beta)$ satisfies \eqref{eq:HypNoDense} for all $T>0$. In that cases, for example, Theorem~\ref{thm:gapest} and Proposition~\ref{prop:relgap}
hold for arbitrary $T$. Nevertheless, for $n \geq 3$, even if the $L^1$-solution exists and coincides with the $L^\infty$-solution globally in time, $E_T(\beta)$ may violate \eqref{eq:HypNoDense}. 
For $n=3$, an example is also given in \cite{Oberguggenberger1986}. Hence, in general, for $n \geq 3$, \eqref{eq:HypNoDense} must be explicitly checked.
\end{remark}

\section{Application to flux switching control of conservation laws}\label{sec:example}
We consider the integer controlled semilinear hyperbolic system
\begin{equation}\label{eq:burgersrelax}
\begin{aligned}
 \eta_t + \xi_x &= 0,\\
 \xi_t + a^2 \eta_x &= -\kappa^{-1}(\xi-g(\eta,v)), \quad v \in \{1,2\}
\end{aligned}
\end{equation}
with $g(\eta,1)=\frac12 \eta^2$, $g(\eta,2)=-\frac12 \eta^2$, $\kappa>0$ and $a^2$ such that $a^2-\eta^2\geq 0$. 
In characteristic variables $y_1 = \eta + a \xi$ and $y_2 = \eta - a \xi$, the system \eqref{eq:burgersrelax} 
can be written as \eqref{eq:SysPDE} with $\Lambda=\text{diag}(a,-a)$ and a nonlinear function $f(y,u,v)$ independent of $u$.

For sufficiently small $\kappa$, the system \eqref{eq:burgersrelax} is an approximation of the control system 
\begin{equation}\label{eq:burgers}
 \eta_t \pm {\textstyle \frac12} \eta^2_x = 0
\end{equation}
where the control just consists of switching the sign in the flux function of the conservation law \eqref{eq:burgers}.
The approximation holds in the sense that, for fixed $v$, up to second order in $\kappa$
\begin{equation}\label{eq:burgersapprox}
 \eta_t \pm {\textstyle \frac12} \eta^2_x = \kappa((a^2-\eta^2)\eta_x)_x,
\end{equation}
see, \cite{JinXin1995,Bianchini2001}. Such flux switching control problems appear for example when traffic flow modeled by conservation laws 
is supposed to be optimized by switching dynamic speed limit signs as in \cite{HeygiEtAl2005}. Since Burger's equation is a typical test 
problem for traffic flow, this example shall demonstrate that the relaxation method investigated in this paper is a very efficient way 
to solve such problems with the advantage of being supported by a convergence theory and being applicable up to the same discretization levels 
that can be handled by optimal control techniques for hyperbolic systems without integer confinements.

For our example, we consider the initial data $\eta(0,x)=\eta_0(x)$, $\xi(0,x)=\frac12 g(\eta_0,1)+\frac12 g(\eta_0,2)=0$,
for a given $\eta_0$, periodic boundary conditions $\eta(t,0)=\eta(t,L)$, $\xi(t,0)=\xi(t,L)$ at the end points of the interval $[0,L]$ 
and consider the minimization of a tracking type cost functional
\begin{equation}\label{eq:costtrack}
 J(\eta)=\frac12\int_0^L |\eta(T,x)-\bar{\eta}(x)|^2\,dx
\end{equation}
for a given reference solution $\bar{\eta}$.
Assuming $\eta_0$ and $\bar{\eta}$ being piecewise $W^{1,1}$ on $(0,L)$ (which is not restrictive in most applications), we obtain from 
Lemma~\ref{lem:Oberguggenberger} that the $L^1$-solution of \eqref{eq:burgersrelax} is piecewise $W^{1,1}$ on $(0,L)$ for 
each $t \in [0,T]$. Since in this example we have $n=2$, \eqref{eq:HypNoDense} is satisfied for all $T>0$, cf. Remark~\ref{rem:Edense}. 
Hence, for any $T>0$ such that the usual $L^1$-solution exists, $\eta(T,\cdot)$ and $\bar{\eta}$ are, by Sobolev-embeddings, 
also in $L^2(0,L)$ and the evaluation of $J$ is therefore well-posed for any piecewise constant control $v\: [0,T] \to \{1,2\}$.

According to the relaxation method presented at the beginning of Section~\ref{sec:estimates} motivating Proposition~\ref{prop:relgap}, we 
introduce two new controls $\alpha_1,\alpha_2\: [0,T] \to \{0,1\}$, such that $\alpha_1(t)+\alpha_2(t)=1$, and consider their relaxation 
$\beta_1,\beta_2\: [0,T] \to [0,1]$ subject to $\beta_1(t)+\beta_2(t)=1$. Using the latter constraint yields $\beta_2(t)=1-\beta_1(t)$, 
so that the the relaxed control problem \eqref{eq:SysPDErelaxed} contains only a single control function $\beta(t)=\beta_1(t)$. Moreover, 
using that we have $\beta(t)(\xi-g(\eta,1))+(1-\beta(t))(\xi-g(\eta,2))=-\beta(t)\eta^2+\frac12 y^2 + \xi$, the relaxed problem 
\eqref{eq:SysPDErelaxed} reads for this example
\begin{equation}\label{eq:burgersrelaxrelax}
\begin{aligned}
   &  \eta_t + \xi_x = 0,\\
   &  \xi_t + a^2 \eta_x = \kappa^{-1}(\beta(t)\eta^2-{\textstyle \frac12} y^2 - \xi).
\end{aligned}
\end{equation}
The optimization problem that determines $\beta$ in Proposition~\ref{prop:relgap} is
\begin{equation}\label{eq:burgersrelaxrelaxocp}
\begin{aligned}
  \min~J(\eta) &\quad \text{subject to}\\
   &  \eta_t + \xi_x = 0,\\
   &  \xi_t + a^2 \eta_x = \kappa^{-1}(\beta(t)\eta^2- {\textstyle \frac12} y^2 - \xi)\\
   &  \eta(0,x)=\eta_0(x),~\xi(0,x)=0,~\eta(t,0)=\eta(t,L),~\xi(t,0)=\xi(t,L)\\
   &  \beta(t) \in [0,1].
\end{aligned}
\end{equation}
We solve \eqref{eq:burgersrelaxrelaxocp} using an adjoint-equation based gradient decent method. To this end, letting $(p,q)$ being the solution of the 
following adjoint equations
\begin{equation}\label{eq:burgersrelaxadjoint}
\begin{aligned}
   &  -p_t - a^2 q_x = \kappa^{-1} q (2\beta(t)-1)\eta,\\
   &  -q_t - p_x = -\kappa^{-1}q\\
   &  p(T,x)=-(\eta(T,x)-\bar{\eta(x)}),~q(T,x)=0,~p(t,L)=p(t,0),~q(t,L)=q(t,0),
\end{aligned}
\end{equation}
the derivative of the reduced cost function $\tilde J(\beta)=J(\eta(\beta))$ can be obtained as
\begin{equation}\label{eq:costtrackGrad}
 \tilde J'(\beta)=-\int_0^L \frac{q}{\kappa}[g(\eta,2)-g(\eta,1)]\,dx = \int_0^L \frac{q}{\kappa} \eta^2\,dx.
\end{equation}
As in \cite{JinXin1995,BandaHerty2009}, we use a first order finite-volume in space and implicit Euler in time splitting scheme 
for the discretization of \eqref{eq:burgersrelaxrelax} and \eqref{eq:burgersrelaxadjoint}. The integrals in \eqref{eq:costtrack} 
and \eqref{eq:costtrackGrad} are evaluated using the trapezoidal rule.

As a test problem, we consider $L=2\pi$, $T=3$, $\eta_0(x)=2\chi_{(\frac{L}{4},\frac{3L}{4})}(x)$, $x \in [0,L]$, 
with $\chi_\gamma$ denoting the characteristic function of the set $\gamma$ in $[0,L]$, and $\bar{\eta}(x)=1-\sin(x)$, $x \in [0,L]$. 
We discretize the space domain $[0,L]$ by $N_x=300$ cells, set $a = 5$, $\kappa = 1.0e{-}08$ and choose a CFL consistent time 
discretization step size with the CFL constant $\nicefrac{1}{2}$ for the time interval $[0,T]$. The discretization level corresponds
to a MINLP with 4.276.800 unknown real and 7.138 unkown binary variables.

With the above adjoint-based approach, we could easily computed a piecewise constant $\beta^*$ with $\tilde J(\beta^*)=0.086$ up to a 
first order optimality of $1.0e{-}08$ using an interior point method. The numerical results corresponding to Proposition~\ref{prop:relgap} 
are reported in Table~\ref{table:results}. The computed controls, the initial state and the corresponding final time states are plotted 
in Figure~\ref{fig:results1} and~\ref{fig:results2}.

In this example, we observe numerically that the relaxed problem exhibits very nice bang-bang structure. This structure is captured 
by the rounding strategy of Lemma~\ref{lem:SUR} for $\Delta t=0.25$ or smaller. We have observed similar results for other initial data 
and control targets. Hence, for the application to flux switching control of the form \eqref{eq:burgersrelax}, the relaxation method 
shows even better convergence properties than those predicted in Proposition~\ref{prop:relgap}. In particular, in this example, we do 
not observe frequent switching in the optimal integer control. This motivates even further investigation of the structure of solutions
to hyperbolic mixed-integer optimal control problems.

\begin{table}
\begin{center}
\begin{tabular}{c|c|c|c|c}
 $k$ & $\Delta t^k$ & $J(v^k)$ & $|J^* - J(v^k)|$ & $\nicefrac{|J^* - J(v^k)|}{J^*}$\\ \hline
 1 & 1.00 & 0.467 & 0.381 & 4.44 \\
 2 & 0.50 &  0.396 & 0.310 & 3.62\\
 3 & 0.25 & 0.086 & 0.000 & 0.00 \\ 
 4 & 0.125 & 0.086 & 0.000 & 0.00 \\ 
 5 & 0.0075 & 0.086 & 0.000 & 0.00  
\end{tabular}
\end{center}
\caption{Numerical results for the test problem on flux switching control for conservation laws discussed in Section~\ref{sec:example}. \label{table:results}}
\end{table}

\begin{figure}
\begin{center}
\includegraphics[width=1.0\textwidth]{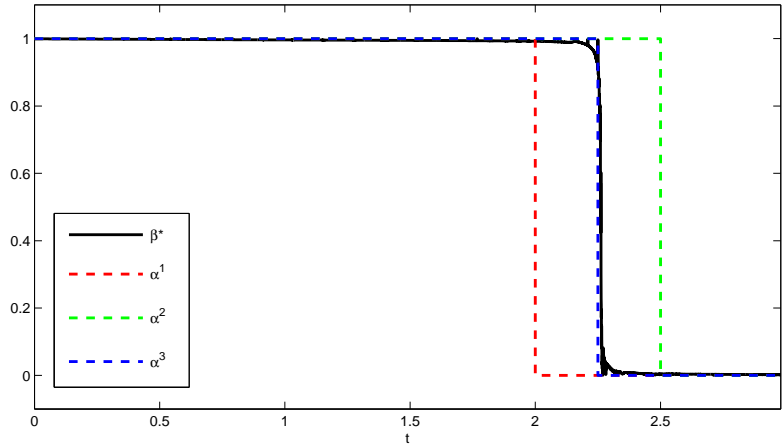}
\end{center}
\caption{Optimal control for the relaxed problem and the integer control obtained by sum up rounding strategies for different $\Delta t$ for the test problem in Section~\ref{sec:example}. \label{fig:results1}}
\end{figure}

\begin{figure}
\begin{center}
\includegraphics[width=1.0\textwidth]{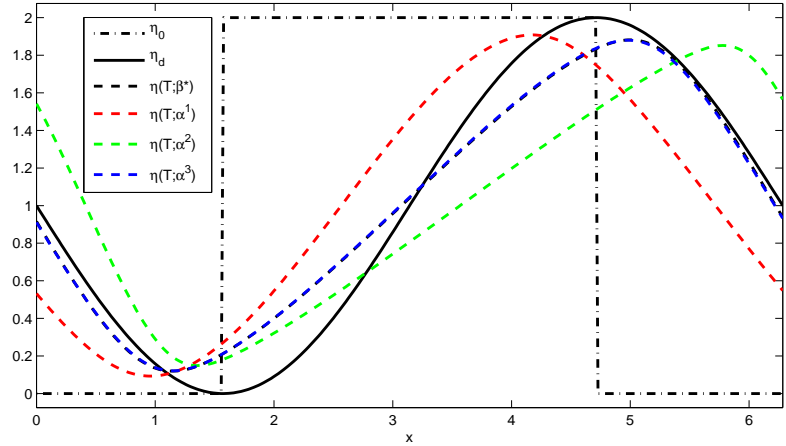}
\end{center}
\caption{Initial data, the desired final time state and finial time plots of the solutions for the relaxed control and different integer controls for the test problem in Section~\ref{sec:example}. \label{fig:results2}}
\end{figure}

\section{Conclusions}\label{sec:conclusion}
Our analysis and numerical results show that certain PDE mixed-integer optimal control problems of hyperbolic type can be solved successfully und very 
efficiently using methods based on relaxation and rounding strategies. 

\section*{Acknowledgments}
This work was supported by the DFG grant CRC/Transregio 154, project A03.

%%%%%%%%

%%%%%%%%

\end{document}